\documentclass[reqno,12pt,letterpaper]{amsart}
\usepackage[proof]{sdmacros}

\DeclareMathOperator{\Sp}{Sp}

\title[Bounds on the fractal uncertainty exponent and a spectral gap]%
{Bounds on the fractal uncertainty exponent\\ and a spectral gap}

\author{Alain Kangabire}
\address{Department of Mathematics, Massachusetts Institute of Technology, 77 Massachusetts Ave, Cambridge, MA 02139}
\email{kanga27@mit.edu}

\begin{document}

\begin{abstract}
We prove two results on Fractal Uncertainty Principle (FUP) for discrete Cantor sets
with large alphabets. First, we give an example of an alphabet with
dimension $\delta\in (\frac12,1)$ where the FUP exponent is exponentially small as
the size of the alphabet grows. Secondly, for $\delta\in (0,\frac12]$ we show
that a similar alphabet has a large FUP exponent, arbitrarily close to
the optimal upper bound of~$\frac12-\frac\delta2$, if we dilate the Fourier transform
by a factor satisfying a generic Diophantine condition. We give
an application of the latter result to spectral gaps for open quantum baker's maps.
\end{abstract}

\maketitle

%%%%%%%%%%%%%%%%%%%%%%%%%%%%%%%%%%%%%%%%%%%%%%%%%%%%%%%%%%%%%%%%%%%%%%%%%%%%%%%%
%%%%%%%%%%%%%%%%%%%%%%%%%%%%%%%%%%%%%%%%%%%%%%%%%%%%%%%%%%%%%%%%%%%%%%%%%%%%%%%%
\section{Introduction}

Open quantum maps have been widely used to model the scattering behavior of open quantum systems,
with relations between the two settings established by Nonnenmacher--Sj\"ostrand--Zworski~\cite{NSZ11}. In this paper, we consider the toy model of quantum open baker's maps 
\begin{equation}
    \label{eqn: quantum open baker's map}
    B_N = \mathcal{F}^*_N \begin{pmatrix}
         \chi_{N/M} \mathcal{F}_{N/M} \chi_{N/M} &  & \\
         & \ddots & \\
         & & \chi_{N/M} \mathcal{F}_{N/M} \chi_{N/M}
    \end{pmatrix}
    I_{\mathcal{A},M}.
\end{equation}
Here $N \in \mathbb N$ is taken to be a multiple of $M \in \mathbb N$, $\mathcal F_N$ is the unitary discrete Fourier transform on $\mathbb C^N$ defined in~\eqref{eqn: discrete fourier} below, and
$\chi_{N/M} \in \mathbb C^{N/M}$ is the discretization of a cutoff function $\chi \in C^{\infty}_{\mathrm c}((0,1);[0,1])$ given by 
\begin{equation}
    \label{eqn: discretization of cutt off}
    \chi_{N/M}(l) = \chi \left( \frac{l}{N/M}\right) \text{ for } 0 \leq l < N/M .
\end{equation}
The set $\mathcal A \subset \mathbb Z_M$, where $\mathbb Z_M = \{0, 1, \dots, M-1\}$, is called an alphabet, and $I_{\mathcal{A},M}$ is a diagonal matrix where the $j^{th}$ diagonal entry is equal to $1$ if $\left \lfloor \frac{j}{N/M} \right \rfloor \in \mathcal{A}$ and 0 otherwise. For example,
the choice $M=3$, $\mathcal A=\{0,2\}$ gives the map
$$
B_N=\mathcal F_N^* \begin{pmatrix} \chi_{N/3}\mathcal F_{N/3}\chi_{N/3} & 0 & 0 \\
0 & 0 & 0\\
0 & 0 & \chi_{N/3}\mathcal F_{N/3}\chi_{N/3}\end{pmatrix}.
$$
Open quantum baker's maps have previously been studied by Keating--Novaes--Prado--Sieber~\cite{KNPS06}, Keating--Nonnenmacher--Novaes--Sieber~\cite{KNNS08}, Nonnenmacher--Zworski~\cite{NZ05,NZ07}, and Dyatlov--Jin~\cite{DJ17, DJ18}; see the introduction to~\cite{DJ17}
for a more detailed review of the literature.
% List: lift from Dyatlov-Jin OQM paper, the papers KNPS, KNNS, Nonnenmacher-Zworski x2, Dyatlov-Jin x2
% List the names next to the papers!

As explained for example in~\cite{NZ05}, we can view $\mathbb C^N$ as
the Hilbert space of quantum states with the classical phase space given by
the torus $\mathbb T^2$ and the semiclassical parameter $h=(2\pi N)^{-1}$. Then the maps $B_N:\mathbb C^N\to\mathbb C^N$
are quantizations of open quantum baker's relations on~$\mathbb T^2$.
The dynamics of these relations naturally gives rise to discrete Cantor sets
\begin{equation}
    \label{def: discrete cantor set}
    \mathcal{C}_k = \{ a_0 + a_1 M + \dots + a_{k-1}M^{k-1} : a_j \in \mathcal{A} \} \subset \mathbb Z_{M^k}
\end{equation}
with dimension
\begin{equation}
    \label{eqn: definition of delta}
    \delta = \frac{\log |\mathcal{A}|}{\log M} ,
\end{equation}
where $|\mathcal{A}|$ is the size of $\mathcal A$.

Dyatlov--Jin~\cite{DJ17} considered the case when $N=M^k$ is a power of $M$ and proved a \emph{spectral gap}, namely an upper bound on the spectral radius of $B_N$:
$$
\limsup_{N\to \infty}\max\{|\lambda|\colon\lambda\in\Sp(B_N)\}\leq M^{-\beta}
$$
where $\Sp(B_N)\subset\mathbb C$ is the spectrum of $B_N$ and $\beta$ is the \emph{fractal uncertainty exponent}, defined as follows:
\begin{equation}
    \label{eqn: fractal exponent}
    \beta = -\lim_{k \rightarrow \infty} \frac{\log \| \mathbbm{1}_{\mathcal{C}_k} \mathcal{F}_{M^k} \mathbbm{1}_{\mathcal{C}_k}\| }{k \log M} .
\end{equation}
Here $\mathbbm{1}_{\mathcal{C}_k}$ is the multiplication operator by the indicator function of the Cantor set $\mathcal C_k$. In other words, $\beta$ is the maximal number such that for any $\epsilon>0$
we have the \emph{fractal uncertainty principle} as $k\to \infty$
\begin{equation}
  \label{eqn: FUP0}
\| \mathbbm{1}_{\mathcal{C}_k} \mathcal{F}_{N} \mathbbm{1}_{\mathcal{C}_k}\|\leq C_\epsilon N^{-\beta+\epsilon},\quad N:=M^k,
\end{equation}
which intuitively says that no function can be localized both in position and frequency close to the fractal set $\mathcal C_k$. It is easy to show that~\eqref{eqn: FUP0} holds with $\beta=0$ and with $\beta=\frac12-\delta$, and~\cite{DJ17} proved that for $0<\delta<1$, one can always improve on these bounds. Moreover, the bound~\eqref{eqn: FUP0} cannot hold with
$\beta>\frac12-\frac\delta 2$. This gives the inequalities
\begin{equation}
  \label{eqn: beta-bound}
\max\left(0, \frac{1}{2} - \delta\right)\ <\ \beta\ \leq\ \frac12-\frac\delta 2.
\end{equation}
This is a model case of more general Fractal Uncertainty Principles which have
found many applications to quantum chaos, see Dyatlov--Zahl~\cite{DyZa16}, Bourgain--Dyatlov~\cite{BD18}, Dyatlov--Jin~\cite{DyJi18}, Dyatlov--Zworski~\cite{DyZw20}, Dyatlov--Jin--Nonnenmacher~\cite{DJN22}, Dyatlov--Jézéquel~\cite{DyJe23}, Cohen~\cite{C22, C23}, Schwartz~\cite{S24}, and Athreya--Dyatlov--Miller~\cite{ADM24}.

% THE LIST:
% Dyatlov-Zahl, Bourgain-Dyatlov Annals, Dyatlov-Zworski IMRN, Dyatlov-Jin Acta, Dyatlov-Jin-Nonnenmacher JAMS, Nir Schwartz cat map, Dyatlov-Jezequel, Athreya-Dyatlov-Miller, Cohen FUP in higher dimensions (both papers)

A natural question is how small the improvement $\beta-\max(0,\frac12-\delta)$ in~\eqref{eqn: beta-bound} can be when $M\to\infty$. In general applications of the Fractal Uncertainty Principle
it is often useful to know what is the largest value of $\beta$ to expect in a given situation,
since it determines for example the size of the spectral gap~\cite{DyZa16, DyZw20} %Dyatlov-Zahl, Dyatlov-Zworski IMRN
or the lower bound on mass of an eigenfunction in a set~\cite{DyJi18, DJN22, DyJe23, S24, ADM24}.
For Cantor sets, this has been studied in~\cite{DJ17}:
\begin{itemize}
\item when $0<\delta<\frac12$, the improvement is bounded above and below by negative powers of~$M$:
the lower bound is given by~\cite[Corollary~3.5]{DJ17} and the upper bound for a specific family of alphabets, by~\cite[Proposition~3.17]{DJ17}.
\item when $\delta=\frac12$, the improvement is bounded above and below by constants times $(\log M)^{-1}$:
the lower bound is given by~\cite[Proposition~3.12]{DJ17} and the upper bound for a specific family of alphabets, again by~\cite[Proposition~3.17]{DJ17}.
\item when $\frac12<\delta<1$, there is a lower bound on the improvement which is exponentially small in~$M$,
given by~\cite[Corollary~3.7]{DJ17}.
\end{itemize}

This leaves out the problem of an upper bound on the improvement in~\eqref{eqn: beta-bound} when $\frac12<\delta<1$. In this case, Dyatlov~\cite[Conjecture 4.4]{D19} conjectured that there exists a sequence of Cantor sets whose fractal uncertainty exponent converges to $0$, with numerical evidence suggesting an exponential convergence with respect to $M$. Hu~\cite[Proposition~3.2]{H21} constructed Cantor sets $\mathcal{C}_k$ with dimension $\frac{1}{2} <\delta <1$ where the fractal uncertainty exponent is bounded by 
\begin{equation*}
    \beta \leq \frac{C}{\log M}
\end{equation*}
for some constant $C$. Additionally, ~\cite[Theorem~3.4]{H21} showed that there exists $\mathcal C_1 \subset \mathbb Z_M$ with dimension $\frac{1}{2} <\delta <1$ and some constant $C$ such that
\begin{equation*}
     \frac{-\log \| \mathbbm{1}_{\mathcal{C}_1} \mathcal{F}_M \mathbbm{1}_{\mathcal{C}_1}\| }{ \log M}\leq \frac{C}{\log M} e^{- \frac{M^{2 \delta -1}}{\log M}} .
\end{equation*}

Our first result is a construction of Cantor sets $\mathcal{C}_k$ with dimension $\frac{1}{2} < \delta < 1$ and having a fractal uncertainty exponent that is exponentially small in $M$:
%%%%%%%%%%%%%%%%%%%%%%%%%%%%%%%%%%%%%%%%%%%%%%%%%%%%%%%%%%%%%%%%%%%%%%%%%%%%%%%%
\begin{theo}
    \label{thm: upper bound on beta for delta bigger than half}
    Fix $\frac{1}{2} < \delta < 1$. There exists $M_0$ such that if $M \geq M_0$ and the alphabet $\mathcal{A} = \left\{l \in \mathbb Z_M: \left|l-\frac{M}{2}\right| \leq \frac{M^{\delta}}{2} \right\}$, then 
    \begin{equation}
        \label{eqn: upper bound on beta for delta bigger than half}
            \beta \leq 170\, e^{-\frac{\pi}{4} M^{2\delta - 1}},
    \end{equation}
    where $\beta$ is defined in \eqref{eqn: fractal exponent}.
    Note that the dimension of the corresponding Cantor set is $\frac{\log(|\mathcal A|)}{\log M} = \delta \pm O\left( \frac{1}{M^{\delta} \log M}\right)$.
\end{theo}
%%%%%%%%%%%%%%%%%%%%%%%%%%%%%%%%%%%%%%%%%%%%%%%%%%%%%%%%%%%%%%%%%%%%%%%%%%%%%%%%
Even though \cite[Proposition 3.17]{DJ17} and Theorem \ref{thm: upper bound on beta for delta bigger than half} show that we can construct Cantor sets with arbitrarily small improvement in~\eqref{eqn: beta-bound}, 
Dyatlov \cite[Conjecture 4.5]{D19} conjectured that if we dilate the Cantor sets $\mathcal C_k$ by a generic factor, then the fractal uncertainty exponent should be much larger. This is
related to gaps for open quantum baker's maps~\eqref{eqn: quantum open baker's map} with more general $N$ than powers of~$M$.
Consider $N \in \mathbb N$ which is just a multiple of $M$. There exist $\alpha \in \mathbb Q$ and $k \in \mathbb N$ such that 
\begin{equation} 
    \label{eqn: range for alpha}
    N = \alpha M^k \text{ and } 1 \leq \alpha < M.
\end{equation}
We define the dilated sets 
\begin{equation}
    \label{def: dilated cantor sets}
    \mathcal C_k(N) = \left\{ \lceil \alpha j \rceil: j \in \mathcal C_k \right\} ,
\end{equation}
where $\alpha$ is given by \eqref{eqn: range for alpha} and $\mathcal C_k$ by \eqref{def: discrete cantor set}.
%The conjecture says that for a generic $\alpha$ obeying \eqref{eqn: range for alpha}, there exist a constant $C$ such that 
%$$
% \left\| \mathbbm{1}_{\mathcal{C}_k(N)} \mathcal{F}_{N} \mathbbm{1}_{\mathcal{C}_k(N)} \right\| \leq C N^{-\beta'} ,
%$$
%where $\beta'$ is better than the fractal uncertainty principle bound $\max \left(0, \frac{1}{2} -\delta \right)$. 

Our second result is a version of~\cite[Conjecture~4.5]{D19}, giving a lower bound
on the fractal uncertainty exponent when $0<\delta\leq \frac12$ and
 $\mathcal C_k$ is the Cantor set that leads to the small fractal uncertainty exponent in \cite[Proposition 3.17]{DJ17} or Theorem \ref{thm: upper bound on beta for delta bigger than half} up to a shift by an integer:
%%%%%%%%%%%%%%%%%%%%%%%%%%%%%%%%%%%%%%%%%%%%%%%%%%%%%%%%%%%%%%%%%%%%%%%%%%%%%%%%  
\begin{theo}
\label{thm: FUP for dilated discrete fourier transform}
For any $\epsilon > 0$, there exist $M_0, k_0 \in \mathbb N$ such that the following holds.
    Let $M\in \mathbb N$ be such that $M\geq M_0$ and $\mathcal A = \left\{0, 1, \dots , M^{\delta} -1 \right\}$, where $M^{\delta} \in \mathbb N$ and $0 < \delta \leq \frac{1}{2}$. Suppose that $\alpha \in \mathbb Q$, $1 \leq \alpha < M$, and $k \in \mathbb N$ are such that $N = \alpha M^k \in M \mathbb Z$ and $k\geq k_0$. Take an irreducible fraction $\frac{b}{q} \in \mathbb Q$ satisfying 
    \begin{equation} 
        \label{eqn: rational approx of alpha over M}
        0 < q \leq M^{\delta} \text{ and } \left| \frac{\alpha}{M} - \frac{b}{q} \right| <\frac{1}{qM^{\delta}}.
    \end{equation}
Then
\begin{equation}
    \label{eqn: FUP for dilated discrete fourier transform}
    \left\| \mathbbm{1}_{\mathcal{C}_k(N)} \mathcal{F}_{N} \mathbbm{1}_{\mathcal{C}_k(N)} \right\| \leq N^{-\left(\frac{1}{2} - \delta + \frac{\gamma}{2}\right) + \epsilon} \;,
\end{equation}
where $\mathcal F_N$ is defined in \eqref{eqn: discrete fourier}, $\mathcal C_k(N)$ is given by \eqref{def: dilated cantor sets}, and $\gamma = \frac{\log q}{\log M} \leq \delta$.
\end{theo}
%%%%%%%%%%%%%%%%%%%%%%%%%%%%%%%%%%%%%%%%%%%%%%%%%%%%%%%%%%%%%%%%%%%%%%%%%%%%%%%%
Note that the existence of $\frac bq$ satisfying~\eqref{eqn: rational approx of alpha over M}
follows from the Dirichlet approximation theorem.

The improvement $\frac\gamma2$ in~\eqref{eqn: FUP for dilated discrete fourier transform}
is large if $\alpha$ is poorly approximable by rationals, which happens for generic $\alpha$.
More precisely, for any $q\in [1,M^\delta]$ the number of $\alpha\in M^{1-k}[M^{k-1},M^k]$
such that~\eqref{eqn: rational approx of alpha over M} holds is bounded above by
$6M^{k-\delta}$. This implies that for any $\epsilon>0$, all but
a $6M^{-\epsilon}\xrightarrow{M\to\infty} 0$ proportion of the allowed values of $\alpha$
satisfy~\eqref{eqn: rational approx of alpha over M}
for some $q\geq M^{\delta-\epsilon}$. For those $\alpha$, the value of $\gamma$ in Theorem~\ref{thm: FUP for dilated discrete fourier transform} is at least $\delta-\epsilon$,
thus~\eqref{eqn: FUP for dilated discrete fourier transform} shows that the fractal uncertainty exponent satisfies
$$
\beta\geq \frac12-\frac\delta2-{3\epsilon\over 2},
$$
which for $\epsilon$ small is close to the theoretical upper bound in~\eqref{eqn: beta-bound}.

As an application of Theorem~\ref{thm: FUP for dilated discrete fourier transform}, we get a spectral gap for baker's maps $B_N$ with size of the matrix $N$ that is just a multiple of $M$. 
%%%%%%%%%%%%%%%%%%%%%%%%%%%%%%%%%%%%%%%%%%%%%%%%%%%%%%%%%%%%%%%%%%%%%%%%%%%%%%%%
\begin{theo}
    \label{thm: spectral gap for baker's maps}
     Fix $\epsilon>0$ and
     assume that $M$ and $\mathcal A$ are as in Theorem \ref{thm: FUP for dilated discrete fourier transform}. Then for all $N\in M\mathbb Z$ we have
          \begin{equation}
         \label{eqn: spectral gap for baker's maps}
      \max \left\{|\lambda|: \lambda \in\Sp(B_N) \right\} \leq M^{-\left(\frac{1}{2} - \delta + \frac{\gamma}{2}\right) + \epsilon} + r(N) ,
     \end{equation}
     where $\delta$ and $\gamma$ are as in Theorem \ref{thm: FUP for dilated discrete fourier transform} and $r(N)\to 0$ as $N\to\infty$.
\end{theo}
%%%%%%%%%%%%%%%%%%%%%%%%%%%%%%%%%%%%%%%%%%%%%%%%%%%%%%%%%%%%%%%%%%%%%%%%%%%%%%%%
Theorem \ref{thm: spectral gap for baker's maps} follows from an argument in \cite[Section 5.1]{DJ18} where Dyatlov--Jin showed that the discrete fractal uncertainty principle \eqref{eqn: FUP for dilated discrete fourier transform} implies the spectral gap \eqref{eqn: spectral gap for baker's maps} for corresponding baker's maps.

\bigskip
\noindent\textbf{Acknowledgements.}
I would like to thank Semyon Dyatlov for invaluable discussions and Alex Cohen for suggesting to look for a function $u$ of the form \eqref{def: u as convolution}. This work was supported by Semyon Dyatlov's NSF CAREER grant DMS-1749858.

\section{Preliminaries} 

We will use the discrete Fourier transform
\begin{equation}
    \label{eqn: discrete fourier}
    \mathcal{F}_N(u)(j) = \frac{1}{\sqrt{N}} \sum_{l = 0}^{N-1} \exp{ \left(-\frac{2\pi i j l}{N}\right)} u(l) ,
\end{equation}
 which is unitary with respect to the norm  
 \begin{equation}
    \label{eqn: definition of the norm of function on C_N}
    \|u\|^2 = \sum_{l = 0}^{N-1} |u(l)|^2 .
\end{equation}
The proof of Theorem \ref{thm: upper bound on beta for delta bigger than half} relies on finding a unit norm function $u_k \in \mathbb C^{M^k}$ such that $\| \mathbbm{1}_{\mathcal{C}_k} \mathcal{F}_{M^k} \mathbbm{1}_{\mathcal{C}_k} u_k \|$ is sufficiently large. 

We construct $u_k$ as a convolution of carefully chosen functions. Given an alphabet $\mathcal{A} \subset \mathbb Z_M$, let $f \in \mathbb C^M$ be a function on $\mathbb Z_M$ with $\supp f \subset \mathcal{A}$. For $ 1 \leq j < k$, define functions $f_{M^j} \in \mathbb{C}^{M^k}$ by  
\begin{equation}
    \label{def: dilated g by powers of M}
    f_{M^j}(n) = 
    \begin{cases}
         f\left(\frac{n}{M^j}\right) &\text{ if } n = M^j a \text{ for some } a \in \mathcal{A}\\
        0 &\text{ otherwise}
    \end{cases}
\end{equation}
Notice that $\supp f_{M^j} \subset M^j \mathcal{A}$. Set $u_1 = f$, and for $k \geq 2$, let
\begin{equation}
    \label{def: u as convolution}
    u_k = f * f_M * f_{M^2} * \cdots * f_{M^{k-2}}*f_{M^{k-1}} ,
\end{equation}
where $*$ is the convolution operator on $\mathbb C^N$, that is 
\begin{equation}
    \label{eqn: definition convolution}
    f*h(n) = \sum_{n_1 + n_2 = n\bmod N} f(n_1) h(n_2) .
\end{equation}
Observe that 
\begin{equation}
    \label{eqn: support of u in C_k}
    \supp u_k \subset \mathcal{C}_k.
\end{equation}
\begin{Remark}
    Sometimes, we will view a function $f$ on $\mathbb Z_M$ as a function on $\mathbb Z$ by extending $f$ to $\mathbb Z$ by zero.
\end{Remark}

We find a formula for $\mathcal{F}_{M^k}(u_k)$ similar to formula (3.36) in ~\cite{DJ17}. 
\begin{lemm}
    \label{lemma: discrete fourier transform of u given by convolution}
    Assume that $u_k$ is of the form \eqref{def: u as convolution}. Then 
    \begin{equation}
        \label{eqn: discrete fourier of u given by convolution}
        \mathcal{F}_{M^k}(u_k)(j) = \prod_{r = 1}^k G_f \left(\frac{j}{M^r}\right),
    \end{equation}
    where $j \in \mathbb Z_{M^k}$ and the 1-periodic function 
    \begin{equation}
        \label{def: inverse fourier series of g}
        G_f(x) = \frac{1}{\sqrt{M}} \sum_{l \in \mathbb Z} f(l) e^{-2 \pi i l x}.
    \end{equation}
\end{lemm}
\begin{proof}
The support condition in the definition \eqref{def: dilated g by powers of M} of $f_{M^j}$  implies that
\begin{equation*}
    \begin{split}
        \mathcal{F}_{M^k}(u_k)(j) = \frac{1}{\sqrt{M^k}} \sum_{a_0 \in \mathcal{A}} \dots \sum_{a_{k-1} \in \mathcal{A}} \left(\prod_{r=0}^{k-1}f(a_r) e^{-2\pi i \frac{ja_r}{M^{k-r}}} \right). \\
    \end{split}
\end{equation*}
Additionally, notice that
\begin{equation*}
    \begin{split}
        \prod_{r=0}^{k-1} \left( \frac{1}{\sqrt{M}}  \sum_{a \in \mathcal{A}} f(a) e^{-2\pi i \frac{ja}{M^{k-r}}} \right) &=  \frac{1}{\sqrt{M^k}} \sum_{a_0 \in \mathcal{A}} \cdots \sum_{a_{k-1} \in \mathcal{A}} \left(\prod_{r=0}^{k-1}f(a_r) e^{-2\pi i \frac{ja_r}{M^{k-r}}} \right)\\
        & = \mathcal{F}_{M^k}(u_k)(j) ,
    \end{split}
\end{equation*}
which gives the desired result by recalling that $\supp f \subset \mathcal A$.
\end{proof}

\begin{lemm}
\label{lemma: eqn: L2 norm of f as sum of M equally spaced points of f}
    Let $f$ be a function on $\mathbb Z$ with support in $\mathbb Z_M$. Then for any $y \in \mathbb R$,
    \begin{equation}
    \label{eqn: L2 norm of f as sum of M equally spaced points of f}
     \sum_{j \in \mathbb{Z}_M} \left| G_f\left( \frac{j}{M} + y \right)\right|^2 = \left\| f \right\|^2
\end{equation}
\end{lemm}
\begin{proof}
    From definition \eqref{def: inverse fourier series of g} of $G_f$ and definition  \eqref{eqn: discrete fourier} of the discrete Fourier transform, it follows that
    \begin{equation*}
    G_f\left(\frac{j}{M} \right) = \mathcal{F}_M \left(f\right)(j),
\end{equation*}
for $j \in \mathbb Z_M$. The unitarity of the discrete Fourier transform implies that
$$
\sum_{j \in \mathbb{Z}_M} \left| G_f\left( \frac{j}{M} \right)\right|^2 = \left\| f \right\|^2 \text{.}
$$
The above equality gives relation \eqref{eqn: L2 norm of f as sum of M equally spaced points of f} because for any fixed $y \in \mathbb R$, 
$$
G_f\left( \frac{j}{M} + y \right) = G_{f_y} \left( \frac{j}{M} \right),
$$
where $f_y(l) = e^{-2\pi i l y} f(l)$.
\end{proof}
\begin{Remark}
    If $u_k$ is of the form \eqref{def: u as convolution}, we can use Lemma \ref{lemma: discrete fourier transform of u given by convolution} and Lemma \ref{lemma: eqn: L2 norm of f as sum of M equally spaced points of f} to show that 
    \begin{equation}
        \label{eqn: norm squared of u_k}
        \| u_k \|^2 = \| f \|^{2k}.
    \end{equation}
    The starting point in the proof of equality \eqref{eqn: norm squared of u_k} is to notice that $\|u_k\|^2 = \| \mathcal F_{M^k} u_k\|^2$ by the unitarity of the discrete Fourier transform. The rest of the argument is similar to the proof of Lemma \ref{lemma: lower bound on norm of truncated fourier transform}, with minor modifications.
\end{Remark}

We define the Fourier transform as
\begin{equation}
    \label{eqn: Fourier transform}
    \hat{f}(\xi) = \int_{\mathbb R} e^{-ix \xi} f(x) dx,
\end{equation}
where $f \in L^2(\mathbb R)$.
\section{Discrete Cantor set with small fractal uncertainty exponent}

In this section, we prove Theorem \ref{thm: upper bound on beta for delta bigger than half} by constructing a unit norm function $u_k \in \mathbb{C}^{M^k}$ of the form \eqref{def: u as convolution} such that $\left\| \mathbbm{1}_{\mathcal{C}_k} \mathcal{F}_{M^k} u_k \right\|$ is large. 

Relation \eqref{eqn: support of u in C_k} says that $\supp u_k \subset \mathcal C_k$, so Lemma \ref{lemma: lower bound on norm of truncated fourier transform} gives an initial lower bound on the operator norm $\left\| \mathbbm{1}_{\mathcal{C}_k} \mathcal{F}_{M^k} \mathbbm{1}_{\mathcal{C}_k}  \right\|$.

\begin{lemm}
    \label{lemma: lower bound on norm of truncated fourier transform}
    If $u_k$ is of the form \eqref{def: u as convolution}, then
    \begin{equation} 
        \label{eqn: lower bound on operator norm of truncated fourier transform}
        \left\| \mathbbm{1}_{\mathcal{C}_k} \mathcal{F}_{M^k} u_k \right\|^2 \geq \left( Z_{\mathcal{A}}(f) \right)^k,
    \end{equation}
    where 
    \begin{equation}
        \label{def: minimum of translated FT of g on A}
        Z_{\mathcal{A}}(f) = \min_{0 \leq y \leq \frac{1}{M}} \sum_{a \in \mathcal{A}} \left| G_f\left( \frac{a}{M} + y \right) \right|^2
    \end{equation}
    and $G_f$ is defined in \eqref{def: inverse fourier series of g}.
\end{lemm}

\begin{proof}
    This Lemma follows easily when $k=1$. Assume that $k \geq 2$. Lemma \ref{lemma: discrete fourier transform of u given by convolution} and the decomposition $\mathcal{C}_k = \mathcal{C}_{k-1} + M^{k-1} \mathcal{A}$ give 
    \begin{equation}
        \label{eqn: initial norm of u as convolution as a sum of products}
        \begin{split}
            \left\| \mathbbm{1}_{\mathcal{C}_k} \mathcal{F}_{M^k} u_k \right\|^2 =\sum_{j' \in \mathcal{C}_{k-1}} \sum_{a_{k-1} \in \mathcal{A}} \left(\prod_{r = 1}^k \left| G_f \left(\frac{j' + a_{k-1}M^{k-1}}{M^r}\right) \right|^2 \right). 
        \end{split}
    \end{equation}
    For $1 \leq r \leq k-1$, notice that 
    $$
    \frac{j' + a_{k-1}M^{k-1}}{M^r} = \frac{j'}{M^r} \mod \mathbb{Z}\,.
    $$
   Since $G_f$ is 1-periodic, equality \eqref{eqn: initial norm of u as convolution as a sum of products} becomes  
    \begin{equation*}
        \begin{split}
            \left\| \mathbbm{1}_{\mathcal{C}_k} \mathcal{F}_{M^k} u_k \right\|^2 & =   \sum_{j' \in \mathcal{C}_{k-1}} \left( \sum_{a_{k-1}  \in \mathcal{A}} \left| G_f \left(\frac{a_{k-1}}{M} + \frac{j'}{M^k}\right) \right|^2 \right) \left(\prod_{r = 1}^{k-1} \left| G_f \left(\frac{j'}{M^r}\right) \right|^2 \right) \\
            &\geq Z_{\mathcal{A}}(f) \sum_{j' \in \mathcal{C}_{k-1}}  \prod_{r = 1}^{k-1} \left| G_f \left(\frac{j'}{M^r}\right) \right|^2, 
        \end{split}
    \end{equation*}
    where the last inequality follows from Definition~\eqref{def: minimum of translated FT of g on A} of $Z_{\mathcal{A}}(f)$  and the inequality $0 \leq \frac{j'}{M^k} \leq \frac{1}{M}$ for $j' \in \mathcal C_{k-1}$. 
    
    See that 
    $$
\sum_{j' \in \mathcal{C}_{k-1}}    \prod_{r = 1}^{k-1} \left| G_f \left(\frac{j'}{M^r}\right) \right|^2  = \left\| \mathbbm{1}_{\mathcal{C}_{k-1}} \mathcal{F}_{M^{k-1}} u_{k-1} \right\|^2 ,
    $$
    so by induction we obtain this lemma. 
\end{proof}

Given $\delta \in \left( \frac{1}{2}, 1 \right)$, we want an alphabet $\mathcal A \subset \mathbb Z_M$ of dimension close to $\delta$ and a function $f$ on $\mathbb Z_M$ with $\supp f \subset \mathcal A$ such that the corresponding quantity $Z_{\mathcal A}\left(f\right)$ defined in \eqref{def: minimum of translated FT of g on A} is large. 

Fix $ \frac{1}{2} < \delta < 1$ and choose $M$ large enough such that $M^{\delta} > 2$. Set
\begin{equation}
    \label{def: choice of alphabet A}
    \mathcal{A} = \left\{l \in \mathbb Z_M: \left|l- \frac{M}{2}\right| \leq \frac{M^{\delta}}{2} \right\} \subset \mathbb Z_M,
\end{equation}
and one can check that
\begin{equation}
    \label{eqn: size of specific alphabet}
    \left| |\mathcal A| - M^{\delta} \right| \leq 1.
\end{equation}
As a consequence, a discrete Cantor set $\mathcal C_k$ having as alphabet the set $\mathcal A$ given by \eqref{def: choice of alphabet A} has dimension 
\begin{equation}
    \label{eqn: dimension of specific cantor set}
    \frac{\log(|\mathcal A|)}{\log M} = \delta + \frac{\log\left(1\pm M^{-\delta}\right)}{\log M} = \delta \pm O\left( \frac{1}{M^{\delta} \log M}\right).
\end{equation}

Next, define the following function on $\mathbb Z$
\begin{equation}
   \label{def: truncated discrete gaussian}
   g(l)= \frac{1}{\sqrt{M}} e^{- \frac{\pi(l - M/2)^2}{M}} e^{i\pi l}.
\end{equation}
The function $|g|$ is the restriction to $\mathbb Z$ of a Gaussian centered at $M/2$ and with standard deviation $\sqrt{M}$, up to a constant factor. The 1-periodic function  $G_g(x)$ defined by \eqref{def: inverse fourier series of g} is 
\begin{equation}
    \label{eqn: G_g for discrete gaussian}
    \begin{split}
        G_g(x) &= \frac{1}{\sqrt{M}} \sum_{l \in \mathbb Z} \frac{1}{\sqrt{M}} e^{- \frac{\pi(l - M/2)^2}{M}} e^{-2\pi il \left(x - \frac{1}{2}\right)} \\
        & = \frac{1}{\sqrt{M}}\sum_{k \in \mathbb Z} \int_{\mathbb R} \left( \frac{1}{\sqrt{M}} e^{- \frac{\pi(y - M/2)^2}{M}} e^{-2\pi i y \left(x - \frac{1}{2}\right)} \right) e^{-2\pi iy k} dy\\
        &= \frac{1}{\sqrt{M}} \sum_{k \in \mathbb Z} e^{- \pi M \left(x - \frac{1}{2} + k\right)^2} e^{- \pi i M \left( x - \frac{1}{2} + k \right) },
    \end{split}
\end{equation}
where the second equality follows from Poisson summation formula and the last equality from the computation $\int_{\mathbb R} \frac{1}{\sqrt{M}} e^{-\frac{\pi x^2}{M}} e^{-2\pi i x \xi} dx = e^{-\pi M \xi^2}$. If we disregard the complex phase, then $MG_g(x)$ is the periodization of a Gaussian centered at $\frac{1}{2}$ and with standard deviation $\frac{1}{\sqrt{M}}$, up to a constant factor. Figure \ref{fig: g and fourier series of g} illustrates the functions $g$, $G_g$, and $G_{g\mathbbm{1}_{\mathcal A}}$.

\begin{figure}
     \centering
     \begin{subfigure}[b]{0.45\textwidth}
         \centering
         \includegraphics[width=\textwidth]{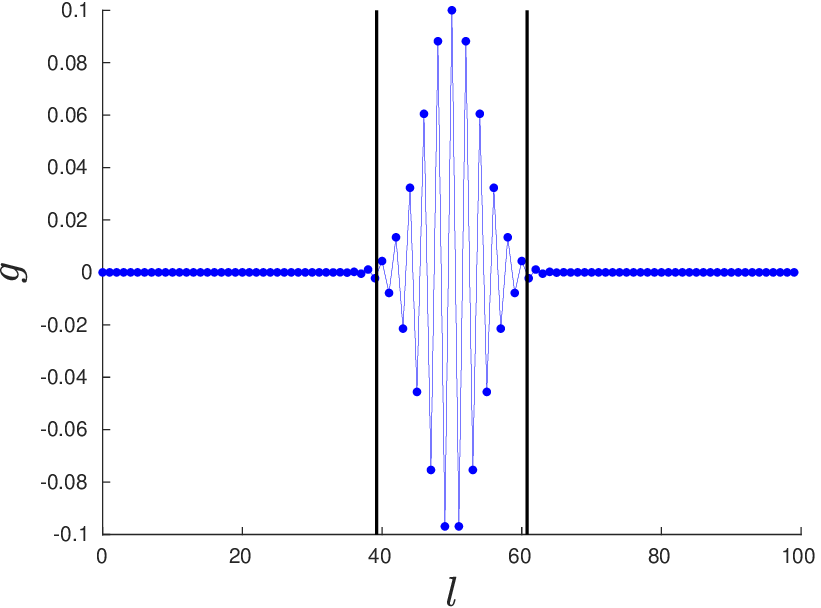}
         \caption{\label{fig: g}}
     \end{subfigure}
     \begin{subfigure}[b]{0.45\textwidth}
         \centering
         \includegraphics[width=\textwidth]{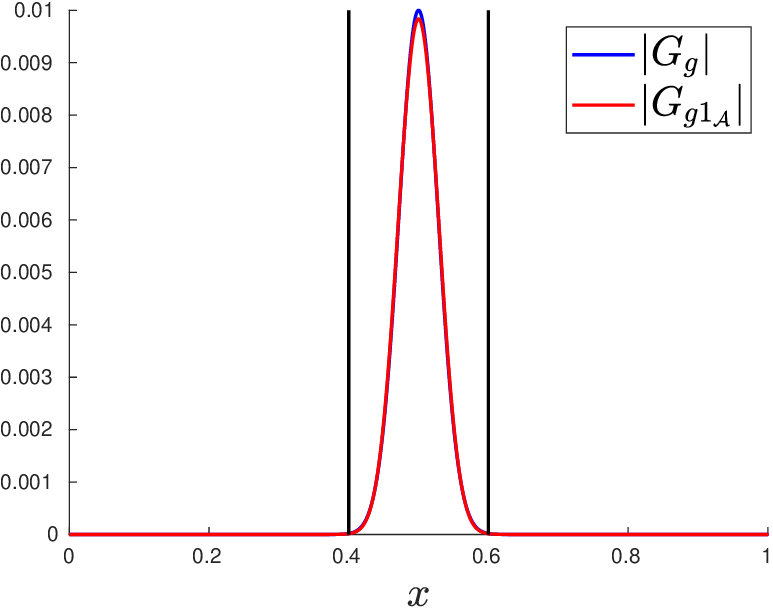}
         \caption{\label{fig: Fourier series of G}}
     \end{subfigure}
        \caption{
        \label{fig: g and fourier series of g}
        We take $M = 100$, $\delta = \frac{2}{3}$, and the alphabet $\mathcal A$ given by \eqref{def: choice of alphabet A}. Figure (A) plots the function $g$ defined in \eqref{def: truncated discrete gaussian}, and the region inside the two vertical black lines represents $g\mathbbm{1}_{\mathcal A}$. Figure (B) shows the functions $G_g$ and $G_{g \mathbbm 1_{\mathcal A}}$ given by \eqref{def: inverse fourier series of g}. The region between the two vertical black lines represents $G_g$ and $G_{g \mathbbm 1_{\mathcal A}}$ restricted on the set $\frac{1}{M}(\mathcal A + [0,1]) \subset [0,1]$.}
\end{figure}

\begin{lemm}
\label{lemma: final bound mass of f outiside subset S}
Let $\mathcal A$ and $g$ be given by \eqref{def: choice of alphabet A} and \eqref{def: truncated discrete gaussian} respectively. If $M$ is large enough and $0 \leq y \leq \frac{1}{M}$, then
    \begin{equation}
        \label{eqn: final bound mass of f outside subset S}
        \sum_{l \in \mathbb Z_M \setminus \mathcal A}  \left|G_{g\mathbbm{1}_{\mathcal A}}\left(\frac{l}{M} + y\right)\right|^2 \leq \frac{60}{\sqrt{M}} e^{-\frac{\pi}{4} M^{2\delta - 1}},
    \end{equation}
    where $G_{g\mathbbm{1}_{\mathcal A}}$ is defined by relation \eqref{def: inverse fourier series of g}.
\end{lemm}
\begin{proof}

 For $x \in [0,1]$, Definition \eqref{def: inverse fourier series of g} of $G_g$ and $G_{g\mathbbm{1}_{\mathcal A}}$ and the triangle inequality give
    \begin{equation}
        \label{eqn: bound on error between inverse Fourier of Gaussian and truncated Gaussian}
        \begin{split}
            \left|G_g(x) - G_{g\mathbbm{1}_{\mathcal A}}(x)\right| &\leq \frac{1}{\sqrt{M}}\sum_{|l-M/2|>\frac{M^{\delta}}{2}} \frac{1}{\sqrt{M}} e^{- \frac{\pi(l - M/2)^2}{M}} \\
            &\leq \frac{1}{\sqrt{M}} \left( 1 + \frac{2}{\sqrt{M}} \right) e^{-\frac{\pi}{4}M^{2\delta - 1}} \\
            &\leq \frac{2}{\sqrt{M}}  e^{-\frac{\pi}{4}M^{2\delta - 1}},
        \end{split}
    \end{equation}
    where the second to last inequality follows from a change of variables ${|l - M/2| - \frac{M^{\delta}}{2}}$ and the upper bound $\sum_{|l|>0}  \frac{1}{\sqrt{M}}e^{-\frac{\pi l^2}{M}} \leq \int_{\mathbb R} \frac{1}{\sqrt{M}} e^{-\frac{\pi x^2}{M}} dx = 1$. We also choose $M\geq 4$ to make $\frac{2}{\sqrt{M}} \leq 1$. 

The inequality $(a+b)^2\leq 2(a^2+b^2)$ and the upper bound \eqref{eqn: bound on error between inverse Fourier of Gaussian and truncated Gaussian} give
\begin{equation} 
    \label{eqn: initial bound mass of f outside subset S}
    \begin{split}
        \sum_{l \in \mathbb Z_M \setminus \mathcal A}  \left|G_{g\mathbbm{1}_{\mathcal A}}\left(\frac{l}{M} + y\right)\right|^2 & \leq \frac{8}{M} e^{-\frac{\pi}{2}M^{2\delta - 1}}  + 2\sum_{l \in \mathbb Z_M \setminus \mathcal A}  \left|G_g\left(\frac{l}{M} + y\right)\right|^2.
    \end{split}
\end{equation}

Take $x \in [0,1]$, then equation \eqref{eqn: G_g for discrete gaussian} for $G_g(x)$ and the triangle inequality give
\begin{equation*}
    |G_g(x)| \leq e^{-\pi M \left( x - \frac{1}{2} \right)^2}\left(\sum_{l \in \mathbb Z} \frac{1}{\sqrt{M}} e^{-\pi M \left(l^2 - |l|\right)} \right),
\end{equation*}
where we used the inequality $l^2 + 2 \left(x -\frac{1}{2}\right)l \geq l^2 - |l|$ for $0 \leq x \leq 1$. When $|l|> 1$, we use inequality $l^2 - |l| \geq \left(|l| - 1\right)^2$ and get
\begin{equation}
    \label{eqn: upper bound on G_g}
    \begin{split}
        |G_g(x)| &\leq e^{-\pi M \left( x - \frac{1}{2} \right)^2} \left( \frac{3}{\sqrt{M}} + \sum_{|l| > 1} \frac{1}{\sqrt{M}} e^{-\pi M \left(|l|-1\right)^2} \right)\\
        &  \leq  \frac{4}{\sqrt{M}}e^{-\pi M \left( x - \frac{1}{2} \right)^2}, 
    \end{split}
\end{equation}
where the last inequality follows by noticing that the infinite sum $\sum_{|l|>0} \sqrt{M} e^{-\pi M l^2}$ is bounded above by ${\int_{\mathbb R} \sqrt{M} e^{-\pi M x^2} dx = 1}$.

The inequality \eqref{eqn: upper bound on G_g} gives
\begin{equation}
    \label{eqn: final bound mass of periodic gaussian outside subset S}
    \begin{split}
        \sum_{l \in \mathbb Z_M \setminus \mathcal A}  \left|G_g \left(\frac{l}{M} + y\right)\right|^2 & \leq  \sum_{|l-M/2| > \frac{M^{\delta}}{2}} \frac{16}{M} e^{-\frac{2 \pi\left(l-M/2 + yM\right)^2}{M}}\\
        & \leq \frac{16 \sqrt{2}}{\sqrt{M}} e^{-\frac{2\pi}{M}\left(\frac{M^{\delta}}{2} - 1\right)^2},
    \end{split}
\end{equation}
with the last inequality following from the same argument used to derive inequality \eqref{eqn: bound on error between inverse Fourier of Gaussian and truncated Gaussian} and the observation that $0 \leq yM \leq 1$.

We substitute in inequality \eqref{eqn: initial bound mass of f outside subset S} the bound obtained in inequality \eqref{eqn: final bound mass of periodic gaussian outside subset S}  to get 
\begin{equation}
    \label{eqn: intermediate bound mass of f outside subset S}
    \begin{split}
        \sum_{l \in \mathbb Z_M \setminus \mathcal A}  \left|G_{g\mathbbm{1}_{\mathcal A}}\left(\frac{l}{M} + y\right)\right|^2 & \leq \frac{8}{M} e^{-\frac{\pi}{2}M^{2\delta - 1}} + \frac{32 \sqrt{2}}{\sqrt{M}} e^{-\frac{2\pi}{M}\left(\frac{M^{\delta}}{2} - 1\right)^2}  \\
        & \leq \frac{60}{\sqrt{M}} e^{-\frac{\pi}{2}\left(1 - \frac{2}{M^{\delta}}\right)^2 M^{2\delta - 1}}\\
        & \leq \frac{60}{\sqrt{M}}  e^{-\frac{\pi}{4} M^{2\delta - 1}}, 
    \end{split}
\end{equation}
where the second to last inequality follows because 
$$
\frac{1}{M}\left(\frac{M^{\delta}}{2} - 1 \right)^2 = \frac{M^{2\delta-1}}{4}\left(1-\frac{2}{M^{\delta}}\right)^2 \leq \frac{M^{2\delta -1}}{4}.
$$ 
In the last inequality, we take $M$ large such that $1 - \frac{2}{M^{\delta}} \geq \frac{1}{2}$.
\end{proof}

Equipped with Lemma \ref{lemma: lower bound on norm of truncated fourier transform} and Lemma \ref{lemma: final bound mass of f outiside subset S}, we are ready to prove Theorem \ref{thm: upper bound on beta for delta bigger than half}.
\begin{proof}[Proof of Theorem \ref{thm: upper bound on beta for delta bigger than half}]
The remark made in Theorem \ref{thm: upper bound on beta for delta bigger than half} that
$$
\frac{\log(|\mathcal A|)}{\log M} = \delta \pm O\left( \frac{1}{M^{\delta} \log M}\right)
$$
is exactly equation \eqref{eqn: dimension of specific cantor set}, so we only need to show the upper bound on $\beta$ given by inequality \eqref{eqn: upper bound on beta for delta bigger than half}.

Let $f = g\mathbbm{1}_{\mathcal A} $, where $g$ is given by equation \eqref{def: truncated discrete gaussian}. Definition \eqref{def: minimum of translated FT of g on A} of $Z_{\mathcal A}( \bullet)$  gives

\begin{equation}
   \label{eqn: lower bound on Z_A(f)}
   \begin{split}
        Z_\mathcal{A}\left(f\right) &= \min_{0 \leq y \leq \frac{1}{M}} \left(   \sum_{l \in \mathbb{Z}_M} \left|G_f\left(\frac{l}{M} + y\right)\right|^2  -    \sum_{l \in \mathbb Z_M \setminus \mathcal{A}} \left|G_f\left(\frac{l}{M} + y\right)\right|^2 \right)\\
        & \geq \|f\|^2 \left(1 - \frac{60}{\|f\|^2 \sqrt{M}} e^{-\frac{\pi}{4} M^{2\delta - 1}} \right)\\
        & \geq \|f\|^2 \left(1 - 170\, e^{-\frac{\pi}{4} M^{2\delta - 1}}\right),
   \end{split}
\end{equation}
where the second to last inequality used Lemma \ref{lemma: eqn: L2 norm of f as sum of M equally spaced points of f} and Lemma \ref{lemma: final bound mass of f outiside subset S}. The last inequality follows from the lower bound $\| f \|^2 \geq \frac{1}{\sqrt{2 M}} \left(\int_{-\frac{M^{\delta}}{2}}^{\frac{M^{\delta}}{2}} \frac{\sqrt{2}}{\sqrt{M}} e^{-\frac{2\pi x^2}{M}} dx - \frac{2 \sqrt{2}}{\sqrt{M}}\right) \geq \frac{1}{2\sqrt{2 M}}$, which is valid when we choose $M$ large enough. 

For $M$ large, the inequality \eqref{eqn: lower bound on Z_A(f)} becomes

\begin{equation}
    \label{eqn: final bound on Z_g_A}
    Z_\mathcal{A}\left(\frac{f}{\left\|f \right\|}\right) \geq  e^{-340\, e^{-\frac{\pi}{4} M^{2\delta - 1}}},
\end{equation}
where we use the inequality $1 - \frac{x}{2} \geq e^{-x}$ for small $x$.

If we use the unit norm function $\frac{f}{\left\|f\right\|}$ in the definition \eqref{def: u as convolution} of $u_k$, then $\|u_k\|^2 = 1$ by equation \eqref{eqn: norm squared of u_k}. Therefore, Lemma \ref{lemma: lower bound on norm of truncated fourier transform} and the lower bound \eqref{eqn: final bound on Z_g_A} give Theorem \ref{thm: upper bound on beta for delta bigger than half}.

\end{proof}

\section{A fractal uncertainty principle for dilated discrete Cantor sets}

We give a proof of Theorem \ref{thm: FUP for dilated discrete fourier transform} in this section. Let $N \in \mathbb N$, $h = (2 \pi N)^{-1}$, and $X$ and $Y$ be subsets of $\mathbb Z_N$. Lemma \ref{lemma: Discrete to continuous FUP} is obtained by arguing similarly to Proposition 5.8 in \cite{DJ18}.

\begin{lemm}
\label{lemma: Discrete to continuous FUP}
    Define the semiclassical Fourier transform 
    \begin{equation}
        \label{def: semiclassical FT}
        \mathcal{F}_h f (\xi)= \frac{1}{\sqrt{2 \pi h}} \int_{\mathbb R} f(x) e^{-\frac{i x \xi}{h}} dx \,.
    \end{equation}
    Then there exists a global constant $C$ such that
    \begin{equation}
    \label{eqn: Discrete to continuous FUP}
        \left\| \mathbbm{1}_{X} \mathcal{F}_N \mathbbm{1}_{Y} \right\| \leq C \left\| \mathbbm{1}_{N^{-1}(X + [0,1])} \mathcal{F}_{h} \mathbbm{1}_{N^{-1}(Y+[0,1])} \right\|_{L^2(\mathbb R) \to L^2(\mathbb R)} \,.
    \end{equation}
\end{lemm}
\begin{proof}
    Let $\mathcal S(\mathbb R)$ be the space of Shwartz functions. We construct $\psi \in \mathcal S(\mathbb R)$ such that its Fourier transform $\hat{\psi}$ defined by \eqref{eqn: Fourier transform} has support in $(-\pi, \pi)$ and $\psi(x) \neq 0$ for all $x \in \mathbb R$. Take $\psi_0 \in \mathcal S (\mathbb R)$ with $\supp \hat{\psi}_0 \subset (-\pi, \pi)$. By the Paley-Wiener theorem, $\psi_0$ extends to a holomorphic function on the complex plane, which has countably many zeros. As a result, there is $y \in \mathbb R$ such that $\psi_0(x + i y) \neq 0$ for every $x \in \mathbb R$. $\psi(x) = \psi_0(x + i y) $ is the desired function, and without loss of generality, we can assume that $|\psi| \geq 1$ on $[-1,1]$.

    Define the following operators
    \begin{equation} 
        \label{def: Fourier multipliers operators}
        \psi(ih \partial) = \mathcal F_h \psi \mathcal F_h^* \text{ , } \psi(-ih\partial) = \mathcal F_h^* \psi \mathcal F_h : L^2(\mathbb R) \rightarrow L^2(\mathbb R),
    \end{equation}
    where $\mathcal F_h^*$ is the adjoint of $\mathcal F_h$ with respect to the $L^2(\mathbb R)$ inner product. A computation shows that the above operators are given by 
    \begin{equation}
        \label{eqn: Fourier multipliers as convolutions}
        \psi(\pm i h \partial)f = E_{\pm} * f \text{,}
    \end{equation}
    where $*$ is the convolution operator and 
    \begin{equation}
        \label{eqn: kernel of Fourier multipliers}
        E_{\pm}(x) = \frac{1}{2\pi h} \hat{\psi}\left( \pm \frac{x}{h} \right).
    \end{equation}
    The condition $\supp \hat{\psi} \subset (-\pi, \pi)$ and equation \eqref{eqn: Fourier multipliers as convolutions} imply that 
    \begin{equation}
        \label{eqn: support enlargement by Fourier multiplier}
        \supp \psi(\pm i h \partial)f \subset \supp f + \left[ -\frac{1}{2N}, \frac{1}{2N} \right].
    \end{equation}

    Let $\mathcal E' (\mathbb R)$ be the space of compactly supported distributions on $\mathbb R$. One can check that 
    \begin{equation}
        \label{eqn: DFT in terms of semiclassical FT}
        \mathcal F_N = \frac{1}{N} T \mathcal F_h T^*,
    \end{equation}
    where $T: C^{\infty}(\mathbb R) \rightarrow \mathbb C^N$ and $T^*: \mathbb C^N \rightarrow \mathcal E' (\mathbb R)$ are given by 
    \begin{equation}
        \label{eqn: Discrete to continuous operators}
        \begin{split}
            Tf(j) &= f\left( \frac{j}{N}\right) \text{ and } T^*f(x) = \sum_{j = 0}^{N-1} f(j) \delta\left(x- \frac{j}{N} \right).
        \end{split}
    \end{equation}
    Using relations \eqref{def: Fourier multipliers operators} and \eqref{eqn: DFT in terms of semiclassical FT}, we have that
    \begin{equation}
        \label{eqn: truncated DFT in terms of semiclassical FT and fourier multipliers}
        \begin{split}
            \mathbbm{1}_{X} \mathcal{F}_N \mathbbm{1}_{Y} &= \frac{1}{N} \mathbbm{1}_{X} T \psi(ih \partial) \psi^{-1} \mathcal F_h \psi(-ih \partial) \psi^{-1} T^* \mathbbm{1}_{Y} \\
             &= \frac{1}{N}  T \psi(ih \partial) \psi^{-1}  \mathbbm{1}_{N^{-1}\left(X + \left[-\frac{1}{2}, \frac{1}{2}\right]\right)} \mathcal F_h \mathbbm{1}_{N^{-1}\left(Y + \left[-\frac{1}{2}, \frac{1}{2}\right]\right)} \psi(-ih \partial) \psi^{-1} T^* ,
        \end{split}
    \end{equation}
    where the last equality follows from the support property \eqref{eqn: support enlargement by Fourier multiplier}. A quick computation shows that 
    
    \begin{equation*}
        \max \left(\left\| \psi(-ih \partial) \psi^{-1} T^* \right\|_{\mathbb C^N \rightarrow L^2(\mathbb R)} \textbf{ , } \left\| T \psi(i h\partial) \psi^{-1} \right\|_{L^2(\mathbb R) \rightarrow \mathbb C^N} \right) \leq \sqrt{N}\| \psi \|_{L^2(\mathbb R)} \, .
    \end{equation*}
    Substituting this upper bound in equation \eqref{eqn: truncated DFT in terms of semiclassical FT and fourier multipliers} gives 
    \begin{equation*}
       \begin{split}
           \left\| \mathbbm{1}_{X} \mathcal{F}_N \mathbbm{1}_{Y} \right\|_{\mathbb C^N \rightarrow \mathbb C^N} &\leq C \left\| \mathbbm{1}_{N^{-1}\left(X + \left[-\frac{1}{2}, \frac{1}{2}\right]\right)} \mathcal{F}_{h} \mathbbm{1}_{N^{-1}\left(Y + \left[-\frac{1}{2}, \frac{1}{2}\right]\right)} \right\|_{L^2(\mathbb R) \rightarrow L^2(\mathbb R)}\\
           & \leq C \left\| \mathbbm{1}_{N^{-1}(X + [0,1])} \mathcal{F}_{h} \mathbbm{1}_{N^{-1}(Y+[0,1])} \right\|_{L^2(\mathbb R) \rightarrow L^2(\mathbb R)} \,,
       \end{split}
    \end{equation*}
    where $C$ can be taken to be $\| \psi \|^2_{L^2(\mathbb R)}$ and the last inequality follows because given subsets $A, B \subset \mathbb R$, the operator norm $\| \mathbbm{1}_A \mathcal F_h \mathbbm{1}_B\|_{L^2(\mathbb R) \rightarrow L^2(\mathbb R)}$ does not change if we replace $A$ and $B$ by the shifted sets $A + a$ and $B + b$, for some $a,b \in \mathbb R$. 
\end{proof}

In Lemma \ref{lemma: upper bound on truncated dilated fourier transform}, we choose specific sets $X,Y \subset Z_N$.
\begin{lemm}
    \label{lemma: upper bound on truncated dilated fourier transform}
    If $N = \alpha M^k \in \mathbb N$ for $M, k \in \mathbb N$ and $\alpha \in \mathbb Q \cap [1,M)$, then there exists a global constant $C$ such that 
    \begin{equation}
    \label{eqn: bound on truncated DFT by dilated cantor sets}
    \left\| \mathbbm{1}_{\mathcal C_k(N)} \mathcal{F}_N \mathbbm{1}_{\mathcal C_k(N)} \right\|^2 \leq \alpha C G^k ,
    \end{equation}
    where $\mathcal C_k(N)$ is defined by \eqref{def: dilated cantor sets},
    \begin{equation}
    \label{eqn: average of F_1 on translated C_1}
        G = M^{-(1-\delta)} \sup_{x \in \mathbb{R}} \sum_{\eta \in \frac{\alpha}{M} \mathcal A} \left( \sup_{ x + \eta + \left[0,\frac{\alpha (\max{\mathcal A}+1)}{M^2}\right] } |F_1| \right) ,
    \end{equation}
    $\mathcal A$ is the alphabet used to construct $\mathcal C_k(N)$, and $F_k$ is given by 
    \begin{equation}
    \label{def: F_k exponential sum on C_k}
        F_k(x) = M^{-\delta k} \sum_{l \in \mathcal{C}_k} e^{-2 \pi i l x} \text{ .}
    \end{equation}
\end{lemm}
\begin{proof}

    From the definition \eqref{def: dilated cantor sets} of $\mathcal{C}_k(N)$, it is clear that 
    $$
    N^{-1} \left(\mathcal C_k(N) + [0,1] \right) \subset N^{-1} \left(\alpha \mathcal C_k + [0,2] \right) \subset M^{-k} \left( \mathcal C_k + [0,2] \right). 
    $$
    Lemma \ref{lemma: Discrete to continuous FUP} and the previous containment give 
    \begin{equation}
    \label{eqn: apply discrete to continuous FUP}
    \begin{split}
        \left\| \mathbbm{1}_{\mathcal C_k(N)} \mathcal{F}_N \mathbbm{1}_{\mathcal C_k(N)} \right\| &\leq C \left\| \mathbbm{1}_{M^{-k} \left( \mathcal C_k + [0,2] \right)} \mathcal{F}_{h} \mathbbm{1}_{M^{-k} \left( \mathcal C_k + [0,2] \right)} \right\|_{L^2(\mathbb R) \rightarrow L^2(\mathbb 
        R)} \\
        & \leq  4C \left\| \mathbbm{1}_{M^{-k} \left( \mathcal C_k + [0,1] \right)} \mathcal{F}_{h} \mathbbm{1}_{M^{-k} \left( \mathcal C_k + [0,1] \right)} \right\|_{L^2(\mathbb R) \rightarrow L^2(\mathbb 
        R)},
    \end{split}
    \end{equation}
    where the last inequality follows by covering $\mathcal C_k + [0,2]$ with two shifted copies of $\mathcal C_k + [0,1]$ and using the triangle inequality. 
    
    Let $\Omega = M^{-k} \left( \mathcal C_k + [0,1] \right)$. We use Schur's lemma to bound the operator norm  $\left\| \mathbbm{1}_{\Omega} \mathcal{F}_{h} \mathbbm{1}_{\Omega} \mathcal{F}_{h}^* \mathbbm{1}_{\Omega} \right\|_{L^2(\mathbb R) \rightarrow L^2(\mathbb R)}$ as follows:
    \begin{equation}
    \label{eqn: apply schur's lemma}
        \begin{split}
            \left\| \mathbbm{1}_{\Omega} \mathcal{F}_{h} \mathbbm{1}_{\Omega} \right\|_{L^2(\mathbb R) \rightarrow L^2(\mathbb R)}^2 & \leq \frac{1}{2 \pi h} \sup_{x \in \Omega} \int_{\Omega} \left| \widehat{\mathbbm{1}}_{\Omega}\left( \frac{y-x}{h} \right) \right| dy \\
            & \leq \alpha \sup_{x \in \mathcal C_k + [0,1]} \int_0^1 \left(\sum_{j \in \mathcal C_k } \left| \widehat{\mathbbm{1}}_{\Omega}\left( 2 \pi \alpha (j + y -x)\right) \right| \right)dy \\
            &\leq  \alpha \sup_{x \in \mathcal C_k + [-1,1]} \sum_{j \in \mathcal C_k } \left| \widehat{\mathbbm{1}}_{\Omega}\left( 2 \pi \alpha (j-x)\right) \right| ,
        \end{split} 
    \end{equation}
    where $\widehat{\mathbbm{1}}_{\Omega}$ is the Fourier transform of the indicator function of $\Omega$ defined by \eqref{eqn: Fourier transform}. A computation shows that 
    \begin{equation*}
        \begin{split}
            \widehat{\mathbbm{1}}_{\Omega}(\xi) = M^{-(1-\delta)k} F_k \left( \frac{M^{-k} \xi}{2 \pi} \right) \int_0^1 e^{-iM^{-k}x \xi} dx,
        \end{split}
    \end{equation*}
    where $F_k$ is defined by \eqref{def: F_k exponential sum on C_k}.
    Substituting this expression for $\widehat{\mathbbm{1}}_{\Omega}$ in relation \eqref{eqn: apply schur's lemma} and combining the resulting inequality with expression \eqref{eqn: apply discrete to continuous FUP} gives that 
    
    \begin{equation}
    \label{eqn: intermediate bound on FUP dilated cantor sets}
        \left\| \mathbbm{1}_{\mathcal C_k(N)} \mathcal{F}_N \mathbbm{1}_{\mathcal C_k(N)} \right\|^2 \leq 16C^2\alpha M^{-(1-\delta)k} \sup_{x \in \mathcal C_k + [-1,1]} \sum_{j \in \mathcal C_k} \left|F_k\left( \alpha M^{-k} (j -  x) \right) \right| .
    \end{equation}
    
    Define  
    \begin{equation}
    \label{def: S_k average of F_k on shifted cantor set}
    S_k = M^{-(1-\delta)k} \sup_{x \in \mathbb{R}} \sum_{j \in \mathcal{C}_k} \left|F_k\left( x + \alpha M^{-k}j\right) \right|.
\end{equation}

Notice that $S_1 \leq G$, where $G$ is defined in \eqref{eqn: average of F_1 on translated C_1}. Using the decomposition $\mathcal{C}_k = \mathcal{C}_1 + M \mathcal{C}_{k-1}$ in the definition \eqref{def: F_k exponential sum on C_k} of $F_k$, we get the recursive relation 
    \begin{equation}
    \label{eqn: recursive relation of F_k}
        \begin{split}
            F_k(x) = F_1(x)F_{k-1}(Mx) .
        \end{split}
    \end{equation}
    In the expression \eqref{def: S_k average of F_k on shifted cantor set} defining $S_k$, we use $\mathcal{C}_k = \mathcal{C}_{k-1} + M^{k-1}\mathcal{A}$ and the recursive formula \eqref{eqn: recursive relation of F_k} for $F_k$ to get
    \begin{equation*}
        \begin{split}
            S_k \leq M^{-(1-\delta)k} \sup_{x\in \mathbb{R}}\sum_{s \in \mathcal{A}} &\left[ \sup_{l \in \mathcal{C}_{k-1}} \left|F_1\left(x + \alpha M^{-k}l+ \alpha M^{-1}s\right)\right| \right. \\
            & \times \left. \left(\sum_{l \in \mathcal{C}_{k-1} }  \left|F_{k-1}\left(Mx + \alpha M^{-(k-1)}l + \alpha s\right)\right| \right) \right] .
        \end{split}
    \end{equation*}
    Notice that for each $x \in \mathbb{R}$ and $s \in \mathcal A$, 
    \begin{equation*}
        \begin{split}
            \sum_{l \in \mathcal{C}_{k-1} }  \left|F_{k-1}\left(Mx + \alpha M^{-(k-1)}l + \alpha s\right)\right| &\leq \sup_{x \in \mathbb{R}} \sum_{l \in \mathcal{C}_{k-1} }  \left|F_{k-1}\left(x + \alpha M^{-(k-1)}l\right)\right|\\
            & \leq M^{(1-\delta)(k-1)}S_{k-1} \, .
        \end{split}
    \end{equation*}
    Consequently, we get 
    \begin{equation*}
        \begin{split}
            S_k & \leq S_{k-1} \left( M^{-(1-\delta)} \sup_{x\in \mathbb{R}}\sum_{s \in \mathcal{A}} \sup_{l \in \mathcal{C}_{k-1}} \left|F_1\left(x + \alpha M^{-k}l+ \alpha M^{-1}s\right)\right| \right)\\
            & \leq S_{k-1} \left( M^{-(1-\delta)} \sup_{x\in \mathbb{R}}\sum_{s \in \mathcal{A}} \sup_{y \in \left[0,\frac{\alpha (\max{\mathcal A}+1)}{M^2}\right]} \left|F_1\left(x + y + \alpha M^{-1}s\right)\right| \right)\\
            & \leq S_{k-1} G,
        \end{split}
    \end{equation*}
    where the second to last inequality follows because $ \max{\mathcal C_{k-1}} \leq \left(\max{\mathcal A} + 1 \right) M^{k-2}$ and $G$ is given by \eqref{eqn: average of F_1 on translated C_1}. Repeating this argument inductively gives 
    \begin{equation*}
        S_k \leq G^k.
    \end{equation*}
    The above inequality, definition \eqref{def: S_k average of F_k on shifted cantor set} of $S_k$,  and inequality \eqref{eqn: intermediate bound on FUP dilated cantor sets} imply this lemma.
\end{proof}

Next, we get an upper bound for $G$ defined in \eqref{eqn: average of F_1 on translated C_1}. The argument used to get this upper bound in Proposition \ref{prop: upper bound on G} is partly inspired by the proof of Theorem 8.1 in \cite{IK04}. 

\begin{prop}
\label{prop: upper bound on G}
Let $M\in \mathbb N$ and $\mathcal A = \left\{0, 1, \dots , M^{\delta} -1 \right\}$, where $M^{\delta} \in \mathbb N$ and $0 < \delta \leq \frac{1}{2}$. Suppose that $\alpha \in \mathbb Q \cap [1,M)$. If $\frac{b}{q} \in \mathbb Q$ is an irreducible fraction such that 
\begin{equation}
\label{eqn: rational approximation of alpha over M}
    0 < q \leq M^{\delta} \text{ and } \left|\frac{\alpha}{M} - \frac{b}{q}\right| \leq \frac{1}{q M^{\delta}} \text{,}
\end{equation}
then
    \begin{equation}
    \label{eqn: upper bound on G}
        G \leq \frac{12}{M^{1-\delta}} \left( \frac{M^{\delta}}{q} + \log q \right),
    \end{equation}
    where $G$ is given by \eqref{eqn: average of F_1 on translated C_1}.
\end{prop}
\begin{proof}
    We get an explicit expression for $G$ defined in \eqref{eqn: average of F_1 on translated C_1}, that is,
    \begin{equation}
    \label{eqn: G for specific parameters}
        \begin{split}
            G = \sup_{x \in \mathbb R} & M^{-(1-\delta)} \sum_{j=0}^{M^{\delta}-1} \sup_{I_j^x} |F_1| ,\\
            \text{where } F_1(x)  = \frac{1}{M^{\delta}} \sum_{j=0}^{M^{\delta}-1} e^{-2\pi i  j x}   &\text{ and } I_j^x = x + \frac{\alpha}{M} j + \left[0, \frac{\alpha M^{\delta}}{M^2}\right]. 
        \end{split}
    \end{equation} 
    For $x \in [0,1]$,  
    \begin{equation}
        \label{eqn: upper bound on F_1}
        \begin{split}
            \left|F_1(x)\right| & \leq \frac{1}{M^{\delta}} \left| \frac{\sin \left(\pi M^{\delta} x \right) }{\sin (\pi x)} \right|\\
            &\leq M^{-\delta} \min (M^{\delta},\| x \|_{\mathbb{Z}}^{-1} ),
        \end{split}
    \end{equation}
    where 
\begin{equation}
    \label{def: distance to closest integer}
    \| x \|_{\mathbb{Z}} = \min_{a \in \mathbb{Z}} |x - a|
\end{equation}
and the last inequality in \eqref{eqn: upper bound on F_1} follows from the lower bound
    $$
    \sin (\pi x) \geq 2x(1-x) \geq \| x \|_{\mathbb{Z}} 
    $$
    and upper bound $|F_1| \leq 1$. 

    Let $x \in \mathbb R$. The second inequality in \eqref{eqn: rational approximation of alpha over M} gives 
\begin{equation}
    \label{eqn: I_x_j intervals around rational points}
    I^x_j = x' + \frac{bj+l}{q} + \gamma + \left[0, \frac{\alpha M^{\delta}}{M^2}\right],
\end{equation}
where $0 \leq j < M^{\delta}$, $| \gamma| \leq \frac{1}{q}$, and $x = \frac{l}{q} + x'$ for some $l \in \mathbb{Z}$ and $|x'| \leq \frac{1}{q}$.

Using relation \eqref{eqn: I_x_j intervals around rational points} and recalling that $1 \leq \alpha <M$ and $0 < \delta \leq \frac{1}{2}$, we get that 
$$
I^x_j \subset I_{j,l} = \frac{bj + l}{q} + \left[ -\frac{2}{q}, \frac{3}{q}\right].
$$
Let $B = \{0, \pm 1, \pm 2, -3\}$. If $j$ is such that $bj+l  \notin (B + q \mathbb Z)$, then $I_j^x$ and $I_{j,l}$ do not contain an integer. As a result, 
\begin{equation}
    \label{eqn: lower bound on distance of I^x_j from integers}
    \left\|I^x_j \right\|_{\mathbb Z} \geq \left\| I_{j,l} \right\|_{\mathbb Z} \geq  \min \left( \left\| \frac{bj + l -2}{q}\right\|_{\mathbb{Z}}, \left\| \frac{bj + l + 3}{q}\right\|_{\mathbb{Z}}  \right) > 0,
\end{equation}
where for an interval $I \subset \mathbb R$,  
$$
\| I\|_{\mathbb Z} = \min_{y \in I \text{, } a \in \mathbb Z}|y-a|.
$$
Inequalities \eqref{eqn: upper bound on F_1} and \eqref{eqn: lower bound on distance of I^x_j from integers} give
\begin{equation*}
    \begin{split}
        \frac{1}{M^{1-\delta}} \sum_{j=0}^{M^{\delta}-1} \sup_{I_j^x} |F_1| &  \leq \sum_{bj + l \in (B + q \mathbb Z)}\frac{1}{M^{1-\delta}} \\
        &+ \sum_{bj+l  \notin (B + q \mathbb Z)}  \frac{1}{M}\min \left( \left\| \frac{bj + l -2}{q}\right\|_{\mathbb{Z}}, \left\| \frac{bj + l + 3}{q}\right\|_{\mathbb{Z}}  \right)^{-1} \\
        & \leq   \frac{12 M^{\delta}}{qM^{1-\delta}} + \frac{2M^{\delta}}{qM} \sum_{a = 3}^{q - 4}  \min \left( \left\| \frac{a -2}{q}\right\|_{\mathbb{Z}}, \left\| \frac{a + 3}{q}\right\|_{\mathbb{Z}}  \right)^{-1} ,
    \end{split}
\end{equation*}
where the last inequality follows because for each $a \in \{0, \dots, q-1\}$, there are at most $\frac{M^{\delta}}{q} + 1$ elements $j\in \{0, \dots, M^{\delta}-1\}$ such that $bj + l = a \bmod q$. A computation shows that 
$$
\sum_{a = 3}^{q - 4}  \min \left( \left\| \frac{a -2}{q}\right\|_{\mathbb{Z}}, \left\| \frac{a + 3}{q}\right\|_{\mathbb{Z}}  \right)^{-1} \leq 2 q \log q,
$$
which implies that
$$
\frac{1}{M^{1-\delta}} \sum_{j=0}^{M^{\delta}-1} \sup_{I_j^x} |F_1| \leq \frac{12}{M^{1-\delta}} \left( \frac{M^{\delta}}{q} + \log q \right).
$$
This last inequality gives the desired proposition.
\end{proof}

We can now give the proof of Theorem \ref{thm: FUP for dilated discrete fourier transform}.

\begin{proof}[Proof of Theorem \ref{thm: FUP for dilated discrete fourier transform}]
Applying Proposition~\ref{prop: upper bound on G}, we get that
$$
G \leq \frac{M^{\delta}}{qM^{1-\delta}} \left( 12 + \frac{12 q}{M^{\delta}}\log q \right),
$$
where $G$ is defined in \eqref{eqn: average of F_1 on translated C_1}. Let $\epsilon > 0$ be small. Because $q \leq M^{\delta}$, we can choose $M$ large enough such that $12(1 + \log M) \leq M^{\epsilon}$, and consequently,
$$
G \leq  \frac{1}{qM^{1-2\delta -\epsilon}} \,.
$$
Substituting the above upper bound on $G$ in Lemma \ref{lemma: upper bound on truncated dilated fourier transform} and recalling that $N = \alpha M^k$ and $q \leq M^{\delta}$, we get that
$$
\| \mathbbm{1}_{\mathcal{C}_k(N)} \mathcal{F}_N \mathbbm{1}_{\mathcal{C}_k(N)} \|^2 \leq \frac{1}{N^{(1-2\delta + \gamma -2\epsilon)}} \left( \frac{\alpha^{(2-\delta - 2\epsilon)} C}{M^{\epsilon k} } \right) \;, 
$$
where $C$ is some global constant and $\gamma = \frac{\log q}{\log M}$. Since $1 \leq \alpha < M$, we can find $k_0$ large enough such that for any $k \geq k_0$, 
$$
\frac{\alpha^{(2-\delta - 2\epsilon)} C}{M^{\epsilon k} } \leq 1,
$$
which gives Theorem \ref{thm: FUP for dilated discrete fourier transform}.
\end{proof}

%%%%%%%%%%%%%%%%%%%%%%%%%%%%%%%%%%%%%%%%%%%%%%%%%%%%%%%%%%%%%%%%%%%%%%%%%%%%%%%%
% BIBLIOGRAPHY
%%%%%%%%%%%%%%%%%%%%%%%%%%%%%%%%%%%%%%%%%%%%%%%%%%%%%%%%%%%%%%%%%%%%%%%%%%%%%%%%


\begin{thebibliography}{0}

\bibitem[ADM24]{ADM24} Jayadev Athreya, Semyon Dyatlov, and Nicholas Miller, \textit{Semiclassical measures for complex hyperbolic quotients}, \arXiv{2402.06477}.

\bibitem[BD18]{BD18} Jean Bourgain and Semyon Dyatlov, \textit{Spectral gaps without the pressure condition}, Ann. of Math. \textbf{187}(2018), 825--867.

\bibitem[C22]{C22} Alex Cohen, Fractal uncertainty for discrete 2D Cantor sets, \arXiv{2206.14131}.

\bibitem[C23]{C23} Alex Cohen, Fractal uncertainty in higher dimensions, \arXiv{2305.05022}.

\bibitem[D19]{D19} Semyon Dyatlov,
\textit{An introduction to fractal uncertainty principle}, J. Math. Phys. \textbf{60}(2019), 081505.

\bibitem[DyJ{\'e}23]{DyJe23} Semyon Dyatlov and Malo Jézéquel, \textit{Semiclassical measures for higher-dimensional quantum cat maps}, Ann. Henri Poincaré(2023).

\bibitem[DJ17]{DJ17} Semyon Dyatlov and Long Jin,
\textit{Resonances for open quantum maps and a fractal uncertainty principle}, Comm. Math. Phys. \textbf{354}(2017), 269--316.

\bibitem[DJ18]{DJ18} Semyon Dyatlov and Long Jin,
\textit{Dolgopyat's method and the fractal uncertainty principle}, Analysis \& PDE \textbf{11}(2018), 1457–1485.

\bibitem[DyJi18]{DyJi18} Semyon Dyatlov and Long Jin, \textit{Semiclassical measures on hyperbolic surfaces have full support}, Acta Math. \textbf{220}(2018),  297--339.

\bibitem[DJN22]{DJN22} Semyon Dyatlov, Long Jin, and Stéphane Nonnenmacher, \textit{Control of eigenfunctions on surfaces of variable curvature}, J. Amer. Math. So. \textbf{35}(2022), 361--465.

\bibitem[DyZa16]{DyZa16} Semyon Dyatlov and Joshua Zahl, \textit{Spectral gaps, additive energy, and a fractal uncertainty principle}, Geom. Funct. Anal. \textbf{26}(2016), 1011--1094.

\bibitem[DyZw20]{DyZw20} Semyon Dyatlov and Maciej Zworski, \textit{Fractal Uncertainty for Transfer Operators}, Int. Math. Res. Not. \textbf{2020}(2020),781--812.

\bibitem[H21]{H21} Nicholas Hu, \textit{Fractal uncertainty principles for ellipsephic sets}, UBC Theses and Dissertations(2021).

\bibitem[IK04]{IK04} Henryk Iwaniec and Emmanuel Kowalski, \textit{Analytic number theory}, AMS Colloquium Publications \textbf{53}(2004).

\bibitem[KNNS08]{KNNS08} Jonathan P. Keating, Stéphane Nonnenmacher, Marcel Novaes, and Martin Sieber, \textit{On the resonance eigenstates of an open quantum baker map}, Nonlinearity \textbf{21}(2008), 2591--2624.

\bibitem[KNPS06]{KNPS06} Jonathan P. Keating, Marcel Novaes, Sandra D. Prado, and Martin Sieber, \textit{Semiclassical Structure of Chaotic Resonance Eigenfunctions}, Phys. Rev. Lett. \textbf{97}(2006), 150406.

\bibitem[NSZ11]{NSZ11} Stéphane Nonnenmacher, Johannes Sjöstrand, and Maciej Zworski, 
\textit{From open quantum systems to open quantum maps}, Comm. Math. Phys. \textbf{304}(2011), 1--48.

\bibitem[NZ05]{NZ05} Stéphane Nonnenmacher and Maciej Zworski, \textit{Fractal Weyl laws in discrete models of chaotic scattering}, J. Phys. A \textbf{38}(2005), 10683--10702.

\bibitem[NZ07]{NZ07} Stéphane Nonnenmacher and Maciej Zworski, \textit{Distribution of resonances for open quantum maps}, Comm. Math. Phys. \textbf{269}(2007), 311--365.

\bibitem[S24]{S24} Nir Schwartz, \textit{The full delocalization of eigenstates for the quantized cat map}, \arXiv{2103.06633}.




\end{thebibliography}
\end{document}